\documentclass[12pt, reqno]{amsart}
\usepackage{amsmath, amsthm, amscd, amsfonts, amssymb, graphicx, color}
\usepackage[bookmarksnumbered, colorlinks, plainpages]{hyperref}
\hypersetup{colorlinks=true,linkcolor=red, anchorcolor=green, citecolor=cyan, urlcolor=red, filecolor=magenta, pdftoolbar=true}

\newtheorem{theorem}{Theorem}[section]
\newtheorem{lemma}[theorem]{Lemma}
\newtheorem{proposition}[theorem]{Proposition}

\theoremstyle{definition}
\newtheorem{definition}[theorem]{Definition}
\newtheorem{example}[theorem]{Example}

\theoremstyle{remark}

\numberwithin{equation}{section}

\begin{document}
	\setcounter{page}{1}
	\title{Continuous relay fusion frame in Hilbert spaces}
	
	\author{Fakhr-dine Nhari$^{1}$ and Mohamed Rossafi$^{2*}$}

	\address{$^{1}$Laboratory Analysis, Geometry and Applications Department of Mathematics, Faculty Of Sciences, University of Ibn Tofail, Kenitra, Morocco}
	\email{\textcolor[rgb]{0.00,0.00,0.84}{nharidoc@gmail.com}}
	
	\address{$^{2}$LASMA Laboratory Department of Mathematics Faculty of Sciences, Dhar El Mahraz University Sidi Mohamed Ben Abdellah, Fes, Morocco}
    \email{\textcolor[rgb]{0.00,0.00,0.84}{rossafimohamed@gmail.com; mohamed.rossafi@usmba.ac.ma}}

	\subjclass[2010]{Primary 41A58; Secondary 42C15.}
	
	\keywords{continous Frame, continuous fusion frame, continuous relay fusion frame, Hilbert space.}
	
	\date{
		\newline \indent $^{*}$Corresponding author}
	
	\begin{abstract}
	In this paper we introduced the concept of continuous relay fusion frames in Hilbert spaces. And we define the dual frames for continuous relay fusion frames. Finally we study the perturbation probleme of continuous relay fusion frames.
	\end{abstract} \maketitle
	
	\section{introduction}
		The concept of frames in Hilbert spaces has been introduced by Duffin and Schaffer \cite{Duf} in 1952 to study some deep problems in nonharmonic Fourier series, after the fundamental paper \cite{DGM} by Daubechies, Grossman and Meyer, frame theory began to be widely used, particularly in the more specialized context of wavelet frames and Gabor frames.
		
		Continuous frames were proposed by G. Kaiser \cite{KAI} and independently  by Ali, Antoine, and Gazeau \cite{GAZ} to a family indexd by some locally compact space endowed by a Radon measure. Gabrado and Han \cite{GHN} called these frames as the ones associated with measurable spaces.
		
		Let $H, L$ be separable Hilbert spaces and let $B(H,L)$ be the space of all the bounded linear operators from $H$ to $L$. (if $H=L$ we write $B(H)$), let  $(\Omega, \mu)$ a positive measure space.
		
		If $W\subseteq H$ and $V\subseteq L$ are subspaces, then we let $\pi_{W}\in B(H)$ and $P_{V}\in B(L)$ denote the orthogonal projections onto the subspaces $W$ and $V$, respectively.
		
		In this section we briefly recall the definitions of continuous frames and continuous fusion frames in Hilbert space. 
	\begin{definition}\cite{RND}
		Let $H$ be a complex Hilbert space and $(\Omega,\mu)$ be a measure space with positive measure $\mu$. A mapping $F:\Omega\rightarrow H$ is called a continuous frame with respect to $(\Omega, \mu)$, if 
		\begin{itemize}
			\item[(1)] $F$ is weakly-measurable, i.e, for all $f\in H$, $w\rightarrow \langle f,F(w)\rangle$ is a measurable function on $\Omega$.
			\item[(2)] there exist constants $A,B>0$ such that 
			\begin{equation}\label{eq2}
				A\|f\|^{2}\leq \int_{\Omega}|\langle f,F(w)\rangle |^{2}d\mu(w)\leq B\|f\|^{2},\quad \forall f\in H.
			\end{equation}
		\end{itemize}
	The constants $A$ and $B$ are called continuous frame bounds. $F$ is called a tight continuous frame if $A=B$. The mapping $F$ is called bessel if the second inequality in \eqref{eq2} holds. In the case, $B$ is called the bessel constant. If $\mu$ is a counting measure and $\Omega=\mathbb{N}$, $F$ is called a discrete frame.
	\end{definition}
	\begin{definition}\cite{MFA}
		Let $\{W_{w}\}_{w\in\Omega}$ be a family of closed subspaces of Hilbert space $H$ and $(\Omega,\mu)$ be a measure space with positive measure $\mu$ and $\nu:\Omega \rightarrow \mathbb{R}^{+}$. Then $\{W_{w},\nu_{w}\}_{w\in\Omega}$ is called a continuous fusion frame with respect to $(\Omega,\mu)$ and $\nu$, if 
		\begin{itemize}
			\item[(1)] for each $f\in H$, $\{\pi_{W_{w}}f\}_{w\in\Omega}$ is strongly measurable and $\nu$ is a measurable function from $\Omega$ to $\mathbb{R}^{+}$.
			\item[(2)] there are two constants $0<C,D<\infty$ such that 
			\begin{equation*}
				C\|f\|^{2}\leq \int_{\Omega}\nu_{w}^{2}\|\pi_{W_{w}}f\|^{2}d\mu(w)\leq D\|f\|^{2},\quad\forall f\in H.
			\end{equation*}
		\end{itemize}
	Where $\pi_{W_{w}}$ is the orthogonal projection onto the subspace $W_{w}$. We call $C$ and $D$ the lower and upper continuous fusion frame bounds, respectively.
	\end{definition}

	\section{continuous relay fusion frame in Hilbert spaces}
	
	\begin{definition}
		Let $\{K_{w}\}_{w\in\Omega}$ be a sequence of separable Hilbert spaces and $\{W_{w}\}_{w\in\Omega}$ be a family of closed subspaces in $H$ for each $w\in\Omega$. Let $\{V_{w,v}\}_{v\in\Omega_{w}}$ be a family of closed subspaces in $K_{w}$. Let $\{\alpha_{w,v}\}_{w\in\Omega,v\in\Omega_{w}}$ be a family of weights, i.e $\alpha_{w,v}>0$ for each $w\in\Omega$, $v\in\Omega_{w}$, and let $\Lambda_{w}\in B(H,K_{w})$ for each $w\in\Omega$. Then $\{W_{w},V_{w,v},\nu_{w,v}\}_{w\in\Omega,v\in\Omega_{w}}$ is said to be a continuous relay fusion frame, if 
		\begin{itemize}
			\item[(1)] for each $f\in H$, $\{\Lambda_{w}f\}_{w\in\Omega}$ is strongly measurable.
			\item[(2)]  for each $f\in H$, $\{\pi_{W_{w}}f\}_{w\in\Omega}$ is strongly measurable.
			\item[(3)] for each $f\in K_{w}$, $\{P_{V_{w,v}}f\}_{w\in\Omega,v\in\Omega_{w}}$ is strongly measurable.	
			\item[(4)] 	there exist constants $0<A\leq B<\infty$ such that 
			\begin{equation}\label{eq1}
				A\|f\|^{2}\leq \int_{\Omega}\int_{\Omega_{w}}\alpha_{w,v}^{2}\|P_{V_{w,v}}\Lambda_{w}\pi_{W_{w}}f\|^{2}d\mu(v) d\mu(w)\leq B\|f\|^{2}, \quad\forall f\in H.
			\end{equation}
		\end{itemize}
		Where $P_{V_{w,v}}$ is the orthogonal projection onto the subspace $V_{w,v}$. We call $A$ and $B$ the lower and upper continuous relay fusion frame bounds, respectively.	
	
	\end{definition}
	
	\begin{definition}
		$$\biggl(\int_{\Omega}\int_{\Omega_{w}}\oplus V_{w,v}\biggr)_{l^{2}}=\biggl\{\{f_{w,v}\}_{w\in\Omega,v\in\Omega_{w}}, f_{w,v}\in V_{w,v}, \int_{\Omega}\int_{\Omega_{w}}\|f_{w,v}\|^{2}<\infty\biggr\},$$ 
		
		with inner product given by 
		\begin{equation*}
			\langle \{f_{w,v}\}_{w,v},\{g_{w,v}\}_{w,v}\rangle=\int_{\Omega}\int_{\Omega_{w}}\langle f_{w,v},g_{w,v}\rangle d\mu(w)d\mu(v).
		\end{equation*}
	With respect to the pointwise operations is a Hilbert space.
	\end{definition}
	\begin{lemma}\label{lem1}
		Let $\mathcal{R}$ be a bessel relay-fusion sequence in $H$ with bessel bound $B$. Then, for each sequence $\{f_{w,v}\}_{w\in\Omega,v\in\Omega_{w}}$ with $f_{w,v}\in V_{w,v}$, 
		\begin{equation*}
			\int_{\Omega}\int_{\Omega_{w}}\alpha_{w,v}\pi_{W_{w}}\Lambda_{w}^{\ast}f_{w,v}d\mu(v)d\mu(w),
		\end{equation*}
	converges.
	\end{lemma}
	\begin{proof}
		Let $f=\{f_{w,v}\}_{w\in\Omega_{w},v\in\Omega_{w}}\in \biggl(\int_{\Omega}\int_{\Omega_{w}}\oplus V_{w,v}\biggr)_{l^{2}}$ and $g=\int_{\Omega}\int_{\Omega_{w}}\alpha_{w,v}\pi_{W_{w}}\Lambda_{w}^{\ast}f_{w,v}d\mu(v)d\mu(w)$.
		Then we have 
		\begin{align*}
			\|g\|&=\|\int_{\Omega}\int_{\Omega_{w}}\alpha_{w,v}\pi_{W_{w}}\Lambda_{w}^{\ast}f_{w,v}d\mu(w)d\mu(v)\|\\&=\sup_{\|h\|=1}\bigg|\langle \int_{\Omega}\int_{\Omega_{w}}\alpha_{w,v}\pi_{W_{w}}\Lambda_{w}^{\ast}f_{w,v}d\mu(w)d\mu(v),h\rangle \bigg|\\&=\sup_{\|h\|=1}\bigg|\int_{\Omega}\int_{\Omega_{w}}\langle f_{w,v},\alpha_{w,v}P_{V_{w,v}}\Lambda_{w}\pi_{W_{w}}h\rangle d\mu(v)d\mu(w)\bigg|\\&\leq\sup_{\|h\|=1}\bigg(\int_{\Omega}\int_{\Omega_{w}}\|f_{w,v}\|^{2}d\mu(v)d\mu(w))\bigg)^{\frac{1}{2}} \bigg(\int_{\Omega}\int_{\Omega_{w}}\alpha_{w,v}^{2}\|P_{V_{w,v}}\Lambda_{w}\pi_{W_{w}}h\|^{2}d\mu(v)d\mu(w)\bigg)^{\frac{1}{2}}\\&\leq \sqrt{B}\|f\|.
		\end{align*}
	\end{proof}
	\begin{definition}
		Let $\mathcal{R}$ be an relay fusion frame for $H$. Then the analysis operator for $\mathcal{R}$ is defined by 
		\begin{equation*}
			T_{\mathcal{R}}: H\rightarrow \biggl(\int_{\Omega}\int_{\Omega_{w}}\oplus V_{w,v}\biggr)_{l^{2}},\quad with \quad T_{\mathcal{R}}(f)=\{\alpha_{w,v}P_{V_{w,v}}\Lambda_{w}\pi_{W_{w}}f\}_{w\in\Omega,v\in\Omega_{w}},\quad \forall f\in H.
		\end{equation*}
	We call the adjoint $T_{\mathcal{R}}^{\ast}$ of the analysis operator the synthesis operator of $\mathcal{R}$.
	\end{definition}
	\begin{proposition}
		Let $\mathcal{R}$ be an relay fusion frame for $H$. Then 
		\begin{equation*}
			T_{\mathcal{R}}^{\ast}f=\int_{\Omega}\int_{\Omega_{w}}\alpha_{w,v}\pi_{W_{w}}\Lambda_{w}^{\ast}f_{w,v}d\mu(v)d\mu(w),\quad \forall \{f_{w,v}\}_{w\in\Omega,v\in\Omega_{w}}\in \biggl(\int_{\Omega}\int_{\Omega_{w}}\oplus V_{w,v}\biggr)_{l^{2}}.
		\end{equation*}
	\end{proposition}
	\begin{proof}
		Let $g\in H$ and $f=\{f_{w,v}\}_{w\in\Omega,v\in\Omega_{w}}\in \biggl(\int_{\Omega}\int_{\Omega_{w}}\oplus V_{w,v}\biggr)_{l^{2}}$.
		Then
		\begin{align*}
			\langle T_{\mathcal{R}}(g),f\rangle&=\langle \{\alpha_{w,v}P_{V_{w,v}}\Lambda_{w}\pi_{W_{w}}g\}_{w\in\Omega,v\in\Omega_{w}}, \{f_{w,v}\}_{w\in\Omega,v\in\Omega_{w}}\rangle\\&=\int_{\Omega}\int_{\Omega_{w}}\langle g,\alpha_{w,v}\pi_{W_{w}}\Lambda_{w}^{\ast}f_{w,v}\rangle d\mu(v)d\mu(w)\\&=\langle g,T_{\mathcal{R}}^{\ast}(f)\rangle.
		\end{align*}
	\end{proof}
	
	\begin{theorem}
		The following assertions are equivalent:
		\begin{itemize}
			\item[(1)] $\mathcal{R}$ is an relay fusion frame for $H$.
			\item[(2)] $T_{\mathcal{R}}$ is injective and has closed range.
		\end{itemize}
	\end{theorem}
	\begin{proof}
		$(1)\implies (2)$
		We have for each $f\in H$:
		\begin{equation*}
			\int_{\Omega}\int_{\Omega_{w}}\alpha_{w,v}\|P_{V_{w,v}}\Lambda_{w}\pi_{W_{w}}f\|^{2}d\mu(v)d\mu(w)=\|T_{\mathcal{R}}f\|^{2}.
		\end{equation*}
	First we prove that $T_{\mathcal{R}}$ is injective, let $f\in H$ such that $T_{\mathcal{R}}f=0$, since 
	\begin{equation*}
		A\|f\|^{2}\leq \|T_{\mathcal{R}}f\|^{2}, \quad\forall f\in H.
	\end{equation*}
Then $f=0$.

We now show that $T_{\mathcal{R}}$ has closed range. Let $\{T_{\mathcal{R}}(x_{n})\}_{n\in \mathbb{N}}\in Range(T_{\mathcal{R}})$ such that $\lim_{n\rightarrow \infty}T_{\mathcal{R}}(x_{n})=y$. For $n,m\in \mathbb{N}$, we have 
\begin{equation*}
	A\|x_{n}-x_{m}\|^{2}\leq \|T_{\mathcal{R}}(x_{n}-x_{m})\|^{2}.
\end{equation*}
Since $\{T_{\mathcal{R}}(x_{n})\}_{n\in \mathbb{N}}$ is Cauchy sequence in $H$, then $\|x_{n}-x_{m}\|\rightarrow 0$, as $n,m\rightarrow \infty$, therefore the sequence $\{x_{n}\}_{n\in\mathbb{N}}$ is Cauchy and hence there exist $x\in H$ such that $x_{n}\rightarrow x$ as $n\rightarrow \infty$. And we have  
\begin{equation*}
	\|T_{\mathcal{R}}(x_{n}-x)\|^{2}\leq B\|x_{n}-x\|^{2}.
\end{equation*}
Thus $\|T_{\mathcal{R}}x_{n}-T_{\mathcal{R}}x\|\rightarrow 0$ as $n\rightarrow \infty$ implies that $T_{\mathcal{R}}x=y$. It concludes that the range of $T_{\mathcal{R}}$ is closed.

$(2)\implies (1)$ This is obvious.

	\end{proof}
By composing $T_{\mathcal{R}}$ and $T_{\mathcal{R}}^{\ast}$, we obtain the frame operator for $\mathcal{R}$.
\begin{definition}
	Let $\mathcal{R}$ be an relay fusion frame operator $S_{\mathcal{R}}$ for $\mathcal{R}$ is defined by 
	\begin{equation*}
		S_{\mathcal{R}}f=T_{\mathcal{R}}^{\ast}T_{\mathcal{R}}f=\int_{\Omega}\int_{\Omega_{w}}\alpha_{w,v}\pi_{W_{w}}\Lambda_{w}^{\ast}P_{V_{w,v}}\pi_{W_{w}}fd\mu(v)d\mu(w).
	\end{equation*}
\end{definition}	
	
\begin{theorem}
	 $\mathcal{R}$ is a bessel relay fusion sequence in $H$ with bound $B$ if and only if the map 
	\begin{equation*}
		\{f_{w,v}\}_{w\in\Omega,v\in\Omega_{w}}\rightarrow \int_{\Omega}\int_{\Omega}\alpha_{w,v}\pi_{W_{w}}\Lambda_{w}^{\ast}f_{w,v}d\mu(v)d\mu(w)
	\end{equation*}
is a well-defined bounded operator frome $\bigg(\int_{\Omega}\int_{\Omega_{w}}\oplus V_{w,v}\bigg)_{l^{2}}$ to $H$ and its norm is less or equal to $\sqrt{B}$.
\end{theorem}	
	\begin{proof}
		First assume that $\mathcal{R}$ is a bessel relay fusion sequence for $H$ with bound $B$, by lemma \ref{lem1} the $\int_{\Omega}\int_{\Omega_{w}}\alpha_{w,v}\pi_{W_{w}}\Lambda_{w}^{\ast}f_{w,v}d\mu(v)d\mu(w)$ is convergent. Thus $T_{\mathcal{R}}^{\ast}(\{f_{w,v}\}_{w\in\Omega,v\in\Omega_{w}})$ is well defined. A simple calculation as in Lemma \ref{lem1} show that $T_{\mathcal{R}}^{\ast}$ is bounded and $\|T_{\mathcal{R}}^{\ast}\|\leq \sqrt{B}$.
		
		For the oppsite implication, suppose that $T_{\mathcal{R}}^{\ast}$ is well defined and that $\|T_{\mathcal{R}}^{\ast}\|\leq \sqrt{B}$. Then 
		\begin{align*}
			\int_{\Omega}\int_{\Omega_{w}}\alpha_{w,v}^{2}\|P_{w,v}\Lambda_{w}\pi_{W_{w}}f\|^{2}d\mu(v)d\mu(w)&=\int_{\Omega}\int_{\Omega_{w}}\alpha_{w,v}^{2}\langle \pi_{W_{w}}\Lambda_{w}^{\ast}P_{V_{w,v}}\Lambda_{w}\pi_{W_{w}}f,f\rangle d\mu(v)d\mu(w)\\&=\langle T_{\mathcal{R}}^{\ast}\{\alpha_{w,v}P_{V_{w,v}}\Lambda_{w}\pi_{W_{w}}f\}_{w\in\Omega,v\in\Omega_{w}},f\rangle \\&\leq \bigg(\int_{\Omega}\int_{\Omega_{w}}\alpha_{w,v}^{2}\|P_{V_{w,v}}\Lambda_{w}\pi_{W_{w}}f\|^{2}d\mu(v)d\mu(w)\bigg)^{\frac{1}{2}}.\|T_{\mathcal{R}}^{\ast}\|.\|f\|,
		\end{align*}
	so we have 
	\begin{equation*}
		\int_{\Omega}\int_{\Omega_{w}}\alpha_{w,v}^{2}\|P_{V_{w,v}}\Lambda_{w}\pi_{W_{w}}f\|^{2}d\mu(v)d\mu(w)\leq \|T_{\mathcal{R}}^{\ast}\|.\|f\|\leq \sqrt{B}.\|f\|.
	\end{equation*}
	\end{proof}
	\begin{theorem}
		Let $\mathcal{R}$ be an relay fusion frame with bounds $A$ and $B$. Then the frame operator for $\mathcal{R}$ is bounded, positive, self-adjoint, invertible operator on $H$ with 
		\begin{equation*}
			AI_{H}\leq S_{\mathcal{R}} \leq BI_{H}.
		\end{equation*}
	\end{theorem}
	\begin{proof}
		$S_{\mathcal{R}}$ is bounded as a composition of two bounded operators
		\begin{equation*}
			\|S_{\mathcal{R}}\|=\|T_{\mathcal{R}}^{\ast}T_{\mathcal{R}}\|=\|T_{\mathcal{R}}^{\ast}\|^{2}\leq B.
		\end{equation*}
	Since 
	\begin{equation*}
		S_{\mathcal{R}}^{\ast}=\big(T_{\mathcal{R}}^{\ast}T_{\mathcal{R}}\big)^{\ast}=T_{\mathcal{R}}^{\ast}T_{\mathcal{R}}=S_{\mathcal{R}},
	\end{equation*}
the operator $S_{\mathcal{R}}$ is self-adjoint. The inequality \eqref{eq1} means that 
\begin{equation*}
	A\|f\|^{2}\leq \langle S_{\mathcal{R}}f,f\rangle \leq B\|f\|^{2}, \forall f\in H.
\end{equation*}
This shows that 
\begin{equation*}
	AI_{H}\leq S_{\mathcal{R}}\leq BI_{H},
\end{equation*}
and hence $S_{\mathcal{R}}$ is positive, invertible operator on $H$.
	\end{proof}
\begin{theorem}
	Let $\mathcal{R}$ be an relay fusion frame for $H$ with frame operator $S_{\mathcal{R}}$, we have then, for all $f\in H$.
	\begin{align*}
		f&=\int_{\Omega}\int_{\Omega_{w}}\alpha_{w,v}^{2}S_{\mathcal{R}}^{-1}\pi_{W_{w}}\Lambda_{w}^{\ast}P_{V_{w,v}}\Lambda_{w}\pi_{W_{w}}fd\mu(v)d\mu(w)\\&=\int_{\Omega}\int_{\Omega_{w}}\alpha_{w,v}^{2}\pi_{W_{w}}\Lambda_{w}^{\ast}P_{V_{w,v}}\Lambda_{w}\pi_{W_{w}}S_{\mathcal{R}}^{-1}fd\mu(v)d\mu(w)
	\end{align*}
\end{theorem}
\begin{proof}
	We have $S_{\mathcal{R}}$ is invertible, then for all $f\in H$
	\begin{equation*}
		f=S_{\mathcal{R}}^{-1}S_{\mathcal{R}}f=\int_{\Omega}\int_{\Omega_{w}}\alpha_{w,v}^{2}S_{\mathcal{R}}^{-1}\pi_{W_{w}}\Lambda_{w}^{\ast}P_{V_{w,v}}\Lambda_{w}\pi_{W_{w}}fd\mu(v)d\mu(w),
	\end{equation*}
and 
    \begin{equation*}
    	f=S_{\mathcal{R}}S_{\mathcal{R}}^{-1}f=\int_{\Omega}\int_{\Omega_{w}}\alpha_{w,v}^{2}\pi_{W_{w}}\Lambda_{w}^{\ast}P_{V_{w,v}}\Lambda_{w}\pi_{W_{w}}S_{\mathcal{R}}^{-1}fd\mu(v)d\mu(w).
    \end{equation*}
\end{proof}	
\begin{theorem}
	Let $\mathcal{R}=\{W_{w},V_{w,v},\alpha_{w,v}\}_{w\in\Omega, v\in\Omega_{w}}$ and $\mathcal{R^{'}}=\{W_{w}^{'},V_{w,v}^{'},\Lambda_{w}^{'},\alpha_{w,v}^{'}\}_{w\in\Omega,v\in\Omega_{w}}$ be two bessel relay sequence for $H$ with bounds $B$ and $B^{'}$, respectively. Let $T_{\mathcal{R}}$ and $T_{\mathcal{R^{'}}}$ be their analysis operators such that $T_{\mathcal{R^{'}}}^{\ast}T_{\mathcal{R}}=I_{H}$. Then both $\mathcal{R}$ and $\mathcal{R^{'}}$ are relay fusion frames.
\end{theorem}
\begin{proof}
	We have for all $f\in H$
	\begin{align*}
		\|f\|^{4}&=\langle T_{\mathcal{R}}f,T_{\mathcal{R^{'}}}f\rangle\\&\leq \|T_{\mathcal{R}}f\|^{2}\|T_{\mathcal{R^{'}}}f\|^{2}\\&=\biggl(\int_{\Omega}\int_{\Omega_{w}}\alpha_{w,v}^{2}\|P_{V_{w,v}}\Lambda_{w}\pi_{W_{w}}f\|^{2}d\mu(v)d\mu(w)\biggr)\\&\quad\quad\quad.\biggl(\int_{\Omega}\int_{\Omega_{w}}\alpha_{w,v}^{'}\|P_{V^{'}_{w,v}}\Lambda_{w}^{'}\pi_{W_{w}^{'}}f\|^{2}d\mu(v)d\mu(w)\biggr)\\&\leq \biggl(\int_{\Omega}\int_{\Omega_{w}}\alpha_{w,v}^{2}\|P_{V_{w,v}}\Lambda_{w}\pi_{W_{w}}f\|^{2}d\mu(v)d\mu(w)\biggr)B^{'}\|f\|^{2}. 
	\end{align*}
Thus 
\begin{equation*}
	\frac{1}{B^{'}}\|f\|^{2}\leq \int_{\Omega}\int_{\Omega_{w}}\alpha_{w,v}^{2}\|P_{V_{w,v}}\Lambda_{w}\pi_{W_{w}}f\|^{2}d\mu(v)d\mu(w).
\end{equation*}
Similarly we obtain a lower bound for $\mathcal{R^{'}}$.
\end{proof}
\section{Duality of relay fusion frames}
\begin{lemma}\cite{GAV}
	Let $A\in B(H)$ and $V\subseteq H$ be a closed subspace. Then 
	\begin{equation*}
		\pi_{V}A^{\ast}=\pi_{V}A^{\ast}\pi_{\overline{AV}}.
	\end{equation*}
\end{lemma}
\subsection{Global continuous relay dual of continuous relay fusion frames}

Let $\mathcal{R}$ be an continuous relay fusion frame for $H$. We consider global continuous relay space $\mathcal{K}=\big(\int_{\Omega}\oplus K_{w}\big)_{l^{2}}$ and let $\mathcal{F}_{\mathcal{K}}$ be a frame for $\mathcal{K}$, where every $K_{w}$ is local continuous relay space. We use $S_{\mathcal{F}_{\mathcal{K}}}$ to denote the frame operator for $\mathcal{K}$. Let $\widetilde{V_{w,v}}=S_{\mathcal{F}_{\mathcal{K}}}^{-1}V_{w,v}$ and $\widetilde{\Lambda_{w}}=S_{\mathcal{F}_{\mathcal{K}}}^{-1}P_{V_{w,v}}\Lambda_{w}$. We now prove that $\widetilde{\mathcal{R}}=\{W_{w},\widetilde{V_{w,v}},\widetilde{\Lambda_{w}},v_{w,v}\}_{w\in\Omega,v\in\Omega_{w}}$ is an continuous relay fusion frame for $H$ and we call $\widetilde{\mathcal{R}}$ the global continuous relay dual continuous relay fusion frame of $\mathcal{R}$.
\begin{theorem}
	Let $\mathcal{R}$ be an continuous relay fusion frame for $H$. Then $\widetilde{\mathcal{R}}$ is an continuous relay fusion frame for $H$, for all $f\in H$,
	\begin{equation*}
		f=\int_{\Omega}\int_{\Omega_{w}}\alpha_{w,v}^{2}S^{-1}_{\widetilde{\mathcal{R}}}\pi_{W_{w}}\widetilde{\Lambda_{w}}^{\ast}\widetilde{\Lambda_{w}}\pi_{W_{w}}fd\mu(v)d\mu(w)=\int_{\Omega}\int_{\Omega_{w}}\alpha_{w,v}^{2}\pi_{W_{w}}\widetilde{\Lambda_{w}}^{\ast}\widetilde{\Lambda_{w}}\pi_{W_{w}}S^{-1}_{\widetilde{\mathcal{R}}}fd\mu(v)d\mu(w).
	\end{equation*}
\end{theorem}
\begin{proof}
	For each $f\in H$, we have 
	\begin{align*}
		\int_{\Omega}\int_{\Omega_{w}}\alpha_{w,v}^{2}\|P_{S_{\mathcal{F}_{\mathcal{K}}}^{-1}V_{w,v}}S_{\mathcal{F}_{\mathcal{K}}}^{-1}P_{V_{w,v}}\Lambda_{w}\pi_{W_{w}}f\|^{2}d\mu(v)d\mu(w)&=\int_{\Omega}\int_{\Omega_{w}}\alpha_{w,v}^{2}\|S_{\mathcal{F}_{\mathcal{K}}}^{-1}P_{V_{w,v}}\Lambda_{w}\pi_{W_{w}}f\|^{2}d\mu(v)d\mu(w)\\&\leq \|S_{\mathcal{F}_{\mathcal{K}}}^{-1}\|^{2}B\|f\|^{2}.
	\end{align*}
On the other hand, we have 
\begin{align*}
	\int_{\Omega}\int_{\Omega_{w}}\alpha_{w,v}^{2}\|P_{S_{\mathcal{F}_{\mathcal{K}}}^{-1}V_{w,v}}S_{\mathcal{F}_{\mathcal{K}}}^{-1}P_{V_{w,v}}\Lambda_{w}\pi_{W_{w}}f\|^{2}d\mu(v)d\mu(w)&=\int_{\Omega}\int_{\Omega_{w}}\alpha_{w,v}^{2}\|S_{\mathcal{F}_{\mathcal{K}}}^{-1}P_{V_{w,v}}\Lambda_{w}\pi_{W_{w}}f\|^{2}d\mu(v)d\mu(w)\\&\geq \int_{\Omega}\int_{\Omega_{w}}\alpha_{w,v}^{2}\frac{1}{\|S_{\mathcal{F}_{\mathcal{K}}}\|^{2}}\|P_{V_{w,v}}\Lambda_{w}\pi_{W_{w}}f\|^{2}d\mu(v)d\mu(w)\\&\geq \frac{A}{\|S_{\mathcal{F}_{\mathcal{K}}}\|^{2}}\|f\|^{2}.
\end{align*}
Further, since $S_{\widetilde{\mathcal{R}}}$ is invertible, for all $f\in H$ we have 
\begin{align*}
	f&=S_{\widetilde{\mathcal{R}}}^{-1}S_{\widetilde{\mathcal{R}}}f=S_{\widetilde{\mathcal{R}}}S_{\widetilde{\mathcal{R}}}^{-1}f\\&=\int_{\Omega}\int_{\Omega_{w}}\alpha_{w,v}^{2}S_{\widetilde{\mathcal{R}}}^{-1}\pi_{W_{w}}\widetilde{\Lambda_{w}}^{\ast}P_{\widetilde{V_{w,v}}}\widetilde{\Lambda_{w}}\pi_{W_{w}}fd\mu(v)d\mu(w)\\&=\int_{\Omega}\int_{\Omega_{w}}\alpha_{w,v}^{2}S_{\widetilde{\mathcal{R}}}^{-1}\pi_{W_{w}}\Lambda_{w}^{\ast}P_{V_{w,v}}S^{-1}_{\mathcal{F}_{\mathcal{K}}}P_{S^{-1}_{\mathcal{F}_{\mathcal{K}}}V_{w,v}}S^{-1}_{\mathcal{F}_{\mathcal{K}}}P_{V_{w,v}}\Lambda_{w}\pi_{W_{w}}fd\mu(v)d\mu(w)\\&=\int_{\Omega}\int_{\Omega_{w}}\alpha_{w,v}^{2}S_{\widetilde{\mathcal{R}}}^{-1}\pi_{W_{w}}\Lambda_{w}^{\ast}P_{V_{w,v}}S^{-1}_{\mathcal{F}_{\mathcal{K}}}S^{-1}_{\mathcal{F}_{\mathcal{K}}}P_{V_{w,v}}\Lambda_{w}\pi_{W_{w}}fd\mu(v)d\mu(w)\\&=\int_{\Omega}\int_{\Omega_{w}}\alpha_{w,v}^{2}S_{\widetilde{\mathcal{R}}}^{-1}\pi_{W_{w}}\widetilde{\Lambda_{w}}^{\ast}\widetilde{\Lambda_{w}}\pi_{W_{w}}fd\mu(v)d\mu(w)\\&=\int_{\Omega}\int_{\Omega_{w}}\alpha_{w,v}^{2}\pi_{W_{w}}\widetilde{\Lambda_{w}}^{\ast}\widetilde{\Lambda_{w}}\pi_{W_{w}}S_{\widetilde{\mathcal{R}}}^{-1}fd\mu(v)d\mu(w).
\end{align*}

\end{proof}
\subsection{Local continuous relay dual of continuous relay fusion frames}
Let $\widehat{V_{w,v}}=S^{-1}_{w}V_{w,v}$ and $\widehat{\Lambda_{w}}=S^{-1}_{w}P_{w,v}\Lambda_{w}$, where $S_{w}$ denote the frame operators with respect to $K_{w}$ for each $w\in \Omega$ and we call every $S_{w}$ local continuous relay frame operator. we now prove that $\widehat{\mathcal{R}}=\{W_{w},\widehat{V_{w,v}},\widehat{\Lambda_{w}},\alpha_{w,v}\}_{w\in\Omega,v\in\Omega_{w}}$ is also an continuous relay fusion frame for $H$ and we call $\widehat{\mathcal{R}}$ the local continuous relay dual continuous relay fusion frame of $\mathcal{R}$.

\begin{theorem}
	Let $\mathcal{R}$ be an continuous relay fusion frame for $H$. Then $\widehat{\mathcal{R}}$ is an continuous relay fusion frame for $H$ and, for all $f\in H$,
	\begin{equation*}
		f=\int_{\Omega}\int_{\Omega_{w}}\alpha_{w,v}^{2}S^{-1}_{\widehat{\mathcal{R}}}\pi_{W_{w}}\widehat{\Lambda_{w}}^{\ast}\widehat{\Lambda_{w}}\pi_{W_{w}}fd\mu(v)d\mu(w)=\int_{\Omega}\int_{\Omega_{w}}\alpha_{w,v}^{2}\pi_{W_{w}}\widehat{\Lambda_{w}}^{\ast}\widehat{\Lambda_{w}}\pi_{W_{w}}S^{-1}_{\widehat{\mathcal{R}}}fd\mu(v)d\mu(w).
	\end{equation*}
\end{theorem}
\begin{proof}
	For all $f\in H$ we have 
	\begin{align*}
		\int_{\Omega}\int_{\Omega_{w}}\alpha_{w,v}^{2}\|P_{S^{-1}_{w}V_{w,v}}S^{-1}_{w}P_{V_{w,v}}\Lambda_{w}\pi_{W_{w}}f\|^{2}d\mu(v)d\mu(w)&=\int_{\Omega}\int_{\Omega_{w}}\alpha_{w,v}^{2}\|S^{-1}_{w}P_{V_{w,v}}\Lambda_{w}\pi_{W_{w}}f\|^{2}d\mu(v)d\mu(w)\\&\leq \max_{w\in\Omega}\{\|S^{-1}_{w}\|^{2}\}B\|f\|^{2}.
	\end{align*}
On the other hand for each $f\in H$ we have 
\begin{align*}
		\int_{\Omega}\int_{\Omega_{w}}\alpha_{w,v}^{2}\|P_{S^{-1}_{w}V_{w,v}}S^{-1}_{w}P_{V_{w,v}}\Lambda_{w}\pi_{W_{w}}f\|^{2}d\mu(v)d\mu(w)&=\int_{\Omega}\int_{\Omega_{w}}\alpha_{w,v}^{2}\|S^{-1}_{w}P_{V_{w,v}}\Lambda_{w}\pi_{W_{w}}f\|^{2}d\mu(v)d\mu(w)\\&\geq \int_{\Omega}\int_{\Omega_{w}}\alpha_{w,v}^{2}\frac{1}{\|S_{w}\|^{2}}\|P_{V_{w,v}}\Lambda_{w}\pi_{W_{w}}f\|^{2}d\mu(v)d\mu(w)\\&\geq \min_{w\in\Omega}\biggl\{\frac{1}{\|S_{w}\|^{2}}\biggr\}A\|f\|^{2}.
\end{align*}
Since $S_{\widehat{\mathcal{R}}}$ is invertible, for all $f\in H$ we have 
\begin{align*}
	f&=S_{\widehat{\mathcal{R}}}^{-1}S_{\widehat{\mathcal{R}}}f=S_{\widehat{\mathcal{R}}}S_{\widehat{\mathcal{R}}}^{-1}f\\&=\int_{\Omega}\int_{\Omega_{w}}\alpha_{w,v}^{2}S^{-1}_{\widehat{\mathcal{R}}}\pi_{W_{w}}\widehat{\Lambda_{w}}^{\ast}P_{\widehat{V_{w,v}}}\widehat{\Lambda_{w}}\pi_{W_{w}}d\mu(v)d\mu(w)\\&=\int_{\Omega}\int_{\Omega_{w}}\alpha_{w,v}^{2}S_{\widehat{\mathcal{R}}}^{-1}\pi_{W_{w}}\Lambda_{w}^{\ast}P_{V_{w,v}}S^{-1}_{w}P_{S^{-1}_{w}V_{w,v}}S^{-1}_{w}P_{V_{w,v}}\Lambda_{w}\pi_{W_{w}}fd\mu(v)d\mu(w)\\&=\int_{\Omega}\int_{\Omega_{w}}\alpha_{w,v}^{2}S_{\widehat{\mathcal{R}}}^{-1}\pi_{W_{w}}\widehat{\Lambda_{w}}^{\ast}\widehat{\Lambda_{w}}\pi_{W_{w}}fd\mu(v)d\mu(w)\\&=\int_{\Omega}\int_{\Omega_{w}}\alpha_{w,v}^{2}\pi_{W_{w}}\widehat{\Lambda_{w}}^{\ast}\widehat{\Lambda_{w}}\pi_{W_{w}}S_{\widehat{\mathcal{R}}}^{-1}fd\mu(v)d\mu(w).
\end{align*}
\end{proof}
\subsection{continuous canonical dual of continuous relay fusion frames}
Now let $\mathring{W_{w}}=S^{-1}_{\mathcal{R}}W_{w}$ and $\mathring{\Lambda_{w}}=\Lambda_{w}\pi_{W_{w}}S^{-1}_{\mathcal{R}}$, where $S_{\mathcal{R}}$ is the frame operator for $\mathcal{R}$. we prove that $\mathring{\mathcal{R}}=\{\mathring{W_{w}},V_{w,v},\mathring{\Lambda_{w}},\alpha_{w,v}\}_{w\in\Omega,v\in\Omega_{w}}$ is also an continuous relay fusion frame for $H$ and we call $\mathring{\mathcal{R}}$ the continuous canonical dual continuous relay fusion frame of $\mathcal{R}$ for $H$.

\begin{theorem}
	Let $\mathcal{R}$ be an continuous relay fusion frame for $H$. Then $\mathring{\mathcal{R}}$ is an continuous relay fusion frame for $H$.
\end{theorem}
\begin{proof}
	We have for all $f\in H$
	\begin{align*}
		\int_{\Omega}\int_{\Omega_{w}}\alpha_{w,v}^{2}\|P_{V_{w,v}}\Lambda_{w}\pi_{W_{w}}S^{-1}_{\mathcal{R}}\pi_{S^{-1}_{\mathcal{R}}W_{w}}f\|^{2}d\mu(v)d\mu(w)&=\int_{\Omega}\int_{\Omega_{w}}\alpha_{w,v}^{2}\|P_{V_{w,v}}\Lambda_{w}\pi_{S^{-1}_{\mathcal{R}}}f\|^{2}d\mu(v)d\mu(w)\\&\leq \|S^{-1}_{\mathcal{R}}\|^{2}B\|f\|^{2}.
	\end{align*}
On the other hand we have 
\begin{align*}
	\|f\|^{4}&=\bigg|\langle \int_{\Omega}\int_{\Omega_{w}}\alpha_{w,v}^{2}\pi_{W_{w}}\Lambda_{w}^{\ast}P_{V_{w,v}}\Lambda_{w}\pi_{W_{w}}S^{-1}_{\mathcal{R}}fd\mu(v)d\mu(w),f\rangle\bigg|^{2}\\&=\bigg|\int_{\Omega}\int_{\Omega_{w}}\alpha_{w,v}^{2}\langle P_{V_{w,v}}\Lambda_{w}\pi_{W_{w}}S^{-1}_{\mathcal{R}}f,P_{V_{w,v}}\Lambda_{w}\pi_{W_{w}}f\rangle d\mu(v)d\mu(w)\bigg|^{2}\\&\leq \biggl(\int_{\Omega}\int_{\Omega_{w}}\alpha_{w,v}^{2}\|P_{V_{w,v}}\Lambda_{w}\pi_{W_{w}}S^{-1}_{\mathcal{R}}\pi_{S^{-1}_{\mathcal{R}}W_{w}}f\|^{2}d\mu(v)d\mu(w)\biggr)\\&\quad\quad\quad .\biggl(\int_{\Omega}\int_{\Omega_{w}}\alpha_{w,v}^{2}\|P_{V_{w,v}}\Lambda_{w}\pi_{W_{w}}f\|^{2}d\mu(v)d\mu(w)\biggr)\\&\leq \biggl(\int_{\Omega}\int_{\Omega_{w}}\alpha_{w,v}^{2}\|P_{V_{w,v}}\Lambda_{w}\pi_{W_{w}}S^{-1}_{\mathcal{R}}\pi_{S^{-1}_{\mathcal{R}}W_{w}}f\|^{2}d\mu(v)d\mu(w)\biggr)B\|f\|^{2}.
\end{align*}
Therefore, 
\begin{equation*}
	\frac{1}{B}\|f\|^{2}\leq \int_{\Omega}\int_{\Omega_{w}}\alpha_{w,v}^{2}\|P_{V_{w,v}}\Lambda_{w}\pi_{W_{w}}S^{-1}_{\mathcal{R}}\pi_{S^{-1}_{\mathcal{R}}W_{w}}f\|^{2}d\mu(v)d\mu(w).
\end{equation*}
\end{proof}
\begin{theorem}
	Let $\mathcal{R}$ be an continuous relay fusion frame for $H$ with frame operator $S_{\mathcal{R}}$ and let $\mathring{\mathcal{R}}$ be the continuous canonical dual continuous relay fusion frame of $\mathcal{R}$ with frame operator $S_{\mathring{\mathcal{R}}}$. Then $S_{\mathcal{R}}S_{\mathring{\mathcal{R}}}=I_{H}$ and $T_{\mathcal{R}}^{\ast}T_{\mathring{\mathcal{R}}}=I_{H}$ and, for all $f\in H$,
	\begin{equation*}
		f=\int_{\Omega}\int_{\Omega_{w}}\alpha_{w,v}^{2}\pi_{W_{w}}\Lambda_{w}^{\ast}P_{V_{w,v}}\mathring{\Lambda_{w}}\pi_{\mathring{W_{w}}}fd\mu(v)d\mu(w)=\int_{\Omega}\int_{\Omega_{w}}\alpha_{w,v}^{2}\pi_{\mathring{W_{w}}}\mathring{\Lambda_{w}}^{\ast}P_{w,v}\Lambda_{w}\pi_{W_{w}}fd\mu(v)d\mu(w).
	\end{equation*}
 \end{theorem}
\begin{proof}
	We have for all $f\in H$,
	\begin{align*}
		S_{\mathcal{R}}S_{\mathring{\mathcal{R}}}f&=S_{\mathcal{R}}\int_{\Omega}\int_{\Omega_{w}}\alpha_{w,v}^{2}\pi_{S^{-1}_{\mathcal{R}}W_{w}}S^{-1}_{\mathcal{R}}\pi_{W_{w}}\Lambda_{w}^{\ast}P_{V_{w,v}}\Lambda_{w}\pi_{W_{w}}S^{-1}_{\mathcal{R}}\pi_{S^{-1}_{\mathcal{R}}W_{w}}fd\mu(v)d\mu(w)\\&=S_{\mathcal{R}}\int_{\Omega}\int_{\Omega_{w}}\alpha_{w,v}^{2}S^{-1}_{\mathcal{R}}\pi_{W_{w}}\Lambda_{w}^{\ast}P_{V_{w,v}}\Lambda_{w}\pi_{W_{w}}S^{-1}_{\mathcal{R}}fd\mu(v)d\mu(w)\\&=\int_{\Omega}\int_{\Omega_{w}}\alpha_{w,v}^{2}\pi_{W_{w}}\Lambda_{w}^{\ast}P_{V_{w,v}}\Lambda_{w}\pi_{W_{w}}S^{-1}_{\mathcal{R}}fd\mu(v)d\mu(w)\\&=S_{\mathcal{R}}S_{\mathcal{R}}^{-1}f\\&=f
	\end{align*}
and 
\begin{align*}
	T_{\mathcal{R}}^{\ast}T_{\mathring{\mathcal{R}}}&=	T_{\mathcal{R}}^{\ast}\bigg(\{\alpha_{w,v}P_{V_{w,v}}\Lambda_{w}\pi_{W_{w}}S^{-1}_{\mathcal{R}}\pi_{S^{-1}_{\mathcal{R}}W_{w}}f\}_{w\in\Omega,v\in\Omega_{w}}\bigg)\\&=	T_{\mathcal{R}}^{\ast}\bigg(\{\alpha_{w,v}P_{V_{w,v}}\Lambda_{w}\pi_{W_{w}}S^{-1}_{\mathcal{R}}f\}_{w\in\Omega,v\in\Omega_{w}}\bigg)\\&=	T_{\mathcal{R}}^{\ast}T_{\mathcal{R}}S^{-1}_{\mathcal{R}}f\\&=f.
\end{align*}
The last assertion of the theorem follows from the previous steps of the proof.
\end{proof}
\begin{theorem}
	Let $\mathcal{R}$ be an continuous relay fusion frame with continuous canonical dual continuous relay fusion frame $\mathring{\mathcal{R}}$. Then, for any $g_{w,v}\in V_{w,v}$ satisfying $f=\int_{\Omega}\int_{\Omega_{w}}\alpha_{w,v}^{2}\pi_{W_{w}}\Lambda_{w}^{\ast}g_{w,v}d\mu(v)d\mu(w)$, we have 
	
	\begin{align*}
		\int_{\Omega}\int_{\Omega_{w}}\|g_{w,v}\|^{2}d\mu(v)d\mu(w)&=\int_{\Omega}\int_{\Omega_{w}}\alpha_{w,v}^{2}\|P_{V_{w,v}}\mathring{\Lambda_{w}}\pi_{\mathring{W_{w}}}f\|^{2}d\mu(v)d\mu(w)\\&\quad\quad+\int_{\Omega}\int_{\Omega_{w}}\|g_{w,v}-\alpha_{w,v}^{2}P_{V_{w,v}}\mathring{\Lambda_{w}}\pi_{\mathring{W_{w}}}f\|^{2}d\mu(v)d\mu(w).
	\end{align*} 
\end{theorem}
\begin{proof}
	For each $f\in H$, we have 
	\begin{align*}
		\int_{\Omega}\int_{\Omega_{w}}\alpha_{w,v}^{2}\|P_{V_{w,v}}\mathring{\Lambda_{w}}\pi_{\mathring{W_{w}}}f\|^{2}d\mu(v)d\mu(w)&=\langle S_{\mathring{\mathcal{R}}}f,f\rangle \\&=\langle f,S_{\mathcal{R}}^{-1}f\rangle\\&=\int_{\Omega}\int_{\Omega_{w}}\alpha_{w,v}^{2}\langle \pi_{W_{w}}\Lambda_{w}^{\ast}g_{w,v},S^{-1}_{\mathcal{R}}f\rangle d\mu(v)d\mu(w) \\&=\int_{\Omega}\int_{\Omega_{w}}\alpha_{w,v}^{2}\langle g_{w,v},P_{w,v}\Lambda_{w}\pi_{W_{w}}S^{-1}_{\mathcal{R}}f\rangle d\mu(v)d\mu(w)\\&=\int_{\Omega}\int_{\Omega_{w}}\alpha_{w,v}^{2}\langle g_{w,v},P_{w,v}\mathring{\Lambda_{w}}\pi_{\mathring{W_{w}}}f\rangle d\mu(v)d\mu(w)\\&=\int_{\Omega}\int_{\Omega_{w}}\alpha_{w,v}^{2}\langle P_{V_{w,v}}\mathring{\Lambda_{w}}\pi_{\mathring{W_{w}}}f,g_{w,v}\rangle d\mu(v)d\mu(w).
	\end{align*}
\end{proof}
\begin{example}
	Let $\mathcal{R}$ be an continuous relay fusion frame for $H$ and $S_{\mathcal{R}}$ denote the frame operator of $\mathcal{R}$. We have for each $f\in H$,
	\begin{align*}
		\|f\|^{2}&=\langle S_{\mathcal{R}}^{-\frac{1}{2}}S_{\mathcal{R}}S_{\mathcal{R}}^{-\frac{1}{2}}f,f\rangle\\&=\langle \int_{\Omega}\int_{\Omega_{w}}\alpha_{w,v}^{2}S_{\mathcal{R}}^{-\frac{1}{2}}\pi_{W_{w}}\Lambda_{w}^{\ast}P_{V_{w,v}}\Lambda_{w}\pi_{W_{w}}S_{\mathcal{R}}^{-\frac{1}{2}}fd\mu(v)d\mu(w),f\rangle\\&=\\&=\int_{\Omega}\int_{\Omega_{w}}\alpha_{w,v}^{2}\|P_{V_{w,v}}\Lambda_{w}\pi_{W_{w}}S_{\mathcal{R}}^{-\frac{1}{2}}f\|^{2}d\mu(v)d\mu(w)\\&=\int_{\Omega}\int_{\Omega_{w}}\alpha_{w,v}^{2}\|P_{V_{w,v}}\Lambda_{w}\pi_{W_{w}}S_{\mathcal{R}}^{-\frac{1}{2}}\pi_{S^{-\frac{1}{2}}_{\mathcal{R}}W_{w}}f\|^{2}d\mu(v)d\mu(w),
	\end{align*}
so, $\{S^{\frac{1}{2}}_{\mathcal{R}}W_{w},V_{w,v},\Lambda_{w}\pi_{W_{w}}S^{\frac{1}{2}}_{\mathcal{R}},\alpha_{w,v}\}_{w\in\Omega,v\in\Omega_{w}}$ is a Parseval continuous relay fusion frame for $H$. 
\end{example}
\begin{theorem}
	Let $\mathcal{R}$ be an continuous relay fusion frame for $H$ with continuous relay fusion frame operator $S_{\mathcal{R}}$ and let $Q\in B(H)$ be an invertible operator. Then $\mathcal{R}_{Q}=\{QW_{w},V_{w,v},\Lambda_{w},\alpha_{w,v}\}_{w\in\Omega,v\in\Omega_{w}}$ is an continuous relay fusion frame for $H$ with continuous relay fusion frame operator $S_{\mathcal{R}_{Q}}$ satisfying 
	\begin{equation*}
		\frac{QS_{\mathcal{R}}Q^{\ast}}{\|Q\|^{2}}\leq S_{\mathcal{R}_{Q}}\leq \|Q^{-1}\|^{2}QS_{\mathcal{R}}Q^{\ast}.
	\end{equation*}
\end{theorem}
\begin{proof}
	For each $f\in H$ we have 
	\begin{equation*}
		\|P_{V_{w,v}}\Lambda_{w}\pi_{W_{w}}Q^{\ast}f\|=\|P_{V_{w,v}}\Lambda_{w}\pi_{W_{w}}Q^{\ast}\pi_{QW_{w}}f\|\leq \|Q^{\ast}\|\|P_{V_{w,v}}\Lambda_{w}\pi_{QW_{w}}f\|.
	\end{equation*}
Since $Q^{\ast}f\in H$ and $\mathcal{R}$ is an continuous relay fusion frame for $H$, we have 
\begin{align*}
	A\|Q^{\ast}f\|^{2}&\leq \int_{\Omega}\int_{\Omega_{w}}\alpha_{w,v}^{2}\|P_{V_{w,v}}\Lambda_{w}\pi_{W_{w}}Q^{\ast}f\|d\mu(v)d\mu(w)\\&\leq\|Q^{\ast}\|^{2}\int_{\Omega}\int_{\Omega_{w}}\alpha_{w,v}^{2}\|P_{V_{w,v}}\Lambda_{w}\pi_{QW_{w}}f\|^{2}d\mu(v)d\mu(w).   
\end{align*}
Thus
\begin{equation*}
	\frac{A}{\|Q^{\ast}\|^{2}\|(Q^{\ast})^{-1}\|^{2}}\|f\|^{2}\leq \int_{\Omega}\int_{\Omega_{w}}\alpha_{w,v}^{2}\|P_{w,v}\Lambda_{w}\pi_{QW_{w}}f\|^{2}d\mu(v)d\mu(w).
\end{equation*}
On the other hand, we have 
\begin{equation*}
	\pi_{QW_{w}}=\pi_{QW_{w}}(Q^{-1})^{\ast}\pi_{W_{w}}Q^{\ast}.
\end{equation*}
So 
\begin{equation*}
	\|P_{V_{w,v}}\Lambda_{w}\pi_{QW_{w}}f\|\leq \|(Q^{-1})^{\ast}\|\|P_{V_{w,v}}\Lambda_{w}\pi_{W_{w}}Q^{\ast}f\|.
\end{equation*}
Therefore
\begin{align*}
	\int_{\Omega}\int_{\Omega_{w}}\alpha_{w,v}^{2}\|P_{V_{w,v}}\Lambda_{w}\pi_{QW_{w}}f\|^{2}d\mu(v)d\mu(w)&\leq \|(Q^{-1})^{\ast}\|^{2}\int_{\Omega}\int_{\Omega_{w}}\alpha_{w,v}^{2}\|P_{V_{w,v}}\Lambda_{w}\pi_{W_{w}}Q^{\ast}f\|^{2}d\mu(v)d\mu(w)\\&\leq B\|Q^{\ast}\|^{2}\|(Q^{\ast})^{-1}\|^{2}\|f\|^{2}.
\end{align*}
Now show that
	\begin{equation*}
	\frac{QS_{\mathcal{R}}Q^{\ast}}{\|Q\|^{2}}\leq S_{\mathcal{R}_{Q}}\leq \|Q^{-1}\|^{2}QS_{\mathcal{R}}Q^{\ast}.
\end{equation*} 
For all $f\in H$, we have 
\begin{align*}
	\langle \frac{QS_{\mathcal{R}}Q^{\ast}}{\|Q\|^{2}}f,f\rangle&=\frac{1}{\|Q\|^{2}}\int_{\Omega}\int_{\Omega_{w}}\alpha_{w,v}^{2}\|P_{V_{w,v}}\Lambda_{w}\pi_{W_{w}}Q^{\ast}f\|^{2}d\mu(v)d\mu(w)\\&=\frac{1}{\|Q\|^{2}}\int_{\Omega}\int_{\Omega_{w}}\alpha_{w,v}^{2}\|P_{V_{w,v}}\Lambda_{w}\pi_{W_{w}}Q^{\ast}\pi_{QW_{w}}f\|^{2}d\mu(v)d\mu(w)\\&\leq \frac{\|Q^{\ast}\|^{2}}{\|Q\|^{2}}\int_{\Omega}\int_{\Omega_{w}}\alpha_{w,v}^{2}\|P_{w,v}\Lambda_{w}\pi_{QW_{w}}f\|^{2}d\mu(v)d\mu(w)\\&=\langle S_{\mathcal{R}_{Q}}f,f\rangle.
\end{align*}
Since 
\begin{equation*}
	\pi_{QW_{w}}=\pi_{QW_{w}}(Q^{-1})^{\ast}\pi_{W_{w}}Q^{\ast}.
\end{equation*}
Then
\begin{align*}
	\langle S_{\mathcal{R}_{Q}}f,f\rangle&=\int_{\Omega}\int_{\Omega_{w}}\alpha_{w,v}^{2}\|P_{V_{w,v}}\Lambda_{w}\pi_{QW_{w}}f\|^{2}d\mu(v)d\mu(w)\\&=\int_{\Omega}\int_{\Omega_{w}}\alpha_{w,v}^{2}\|P_{V_{w,v}}\Lambda_{w}\pi_{QW_{w}}(Q^{-1})^{\ast}\pi_{W_{w}}Q^{\ast}f\|^{2}d\mu(v)d\mu(w)\\&\leq \|(Q^{-1})^{\ast}\|^{2}\int_{\Omega}\int_{\Omega_{w}}\alpha_{w,v}^{2}\|P_{V_{w,v}}\Lambda_{w}\pi_{W_{w}}Q^{\ast}f\|^{2}d\mu(v)d\mu(w)\\&=\|(Q^{-1})^{\ast}\|^{2}\langle S_{\mathcal{R}_{Q}}Q^{\ast}f,Q^{\ast}f\rangle\\&=\langle \|(Q^{-1})^{\ast}\|^{2}QS_{\mathcal{R}_{Q}}Q^{\ast}f,f\rangle
\end{align*}
\end{proof}
\begin{theorem}
	Let $\mathcal{R}$ be an continuous relay fusion frame for $H$ with continuous relay fusion frame bounds $A$ and $B$. If $Q_{w}\in B(K_{w})$ are invertible operators for each $w\in \Omega$, then $\mathcal{R}=\{W_{w},Q_{w}V_{w,v},\Lambda_{w},\alpha_{w,v}\}_{w\in\Omega,v\in\Omega_{w}}$ is an continuous relay fusion frame for $H$.
\end{theorem}
\begin{proof}
	For all $f\in H$, we have 
	\begin{equation*}
		\|Q_{w}P_{V_{w,v}}\Lambda_{w}\pi_{W_{w}}f\|=\|P_{Q_{w}V_{w,v}}Q_{w}P_{V_{w,v}}\Lambda_{w}\pi_{W_{w}}f\|\leq \|Q_{w}\|\|P_{Q_{w}V_{w,v}}\Lambda_{w}\pi_{W_{w}}f\|.
	\end{equation*}
So
\begin{equation*}
	\frac{1}{\|Q_{w}\|\|Q_{w}^{-1}\|}\|P_{V_{w,v}}\Lambda_{w}\pi_{W_{w}}f\|\leq \|P_{Q_{w}V_{w,v}}\Lambda_{w}\pi_{W_{w}}f\|.
\end{equation*}
Therefore
\begin{align*}
	\min_{w\in\Omega}\biggl\{\frac{A}{\|Q_{w}\|^{2}\|Q_{w}^{-1}\|^{2}}\biggr\}\|f\|^{2}&\leq \int_{\Omega}\int_{\Omega_{w}}\frac{\alpha_{w,v}^{2}}{\|Q_{w}\|^{2}\|Q_{w}^{-1}\|^{2}}\|P_{V_{w,v}}\Lambda_{w}\pi_{W_{w}}f\|^{2}d\mu(v)d\mu(w)\\&\leq \int_{\Omega}\int_{\Omega_{w}}\alpha_{w,v}^{2}\|P_{Q_{w}V_{w,v}}\Lambda_{w}\pi_{W_{w}}f\|^{2}d\mu(v)d\mu(w).
\end{align*}
We have 
\begin{equation*}
	P_{Q_{w}V_{w,v}}=P_{Q_{w}V_{w,v}}(Q_{w}^{-1})^{\ast}P_{V_{w,v}}Q_{w}^{\ast}.
\end{equation*}
Then 
\begin{align*}
	\|P_{Q_{w}V_{w,v}}\Lambda_{w}\pi_{W_{w}}f\|&=\|P_{Q_{w}V_{w,v}}(Q_{w}^{-1})^{\ast}P_{V_{w,v}}Q_{w}^{\ast}\Lambda_{w}\pi_{W_{w}}\|\\&\leq \|Q^{\ast}_{w}\|\|(Q_{w}^{-1})^{\ast}\|\|P_{V_{w,v}}\Lambda_{w}\pi_{W_{w}}f\|.
\end{align*}
Hence
\begin{align*}
	&\int_{\Omega}\int_{\Omega_{w}}\alpha_{w,v}^{2}\|P_{Q_{w}V_{w,v}}\Lambda_{w}\pi_{W_{w}}f\|^{2}d\mu(v)d\mu(w)\|^{2}d\mu(v)d\mu(w)\\&\quad\quad\quad\leq\int_{\Omega}\int_{\Omega_{w}}\|Q_{w}^{\ast}\|^{2}\|(Q^{-1}_{w})^{\ast}\|^{2}\alpha_{w,v}^{2}\|P_{V_{w,v}}\Lambda_{w}\pi_{W_{w}}f\|^{2}d\mu(v)d\mu(w)\\&\quad\quad\quad\leq \max_{w\in\Omega}\biggl\{\|Q_{w}^{\ast}\|^{2}\|(Q^{-1}_{w})^{\ast}\|^{2}\biggr\}B\|f\|^{2}
\end{align*}
\end{proof}
\section{Perturbation of the continuous relay fusion frames }
\begin{theorem}
	Let $\mathcal{R}_{1}=\{W_{w},V_{w,v},\Lambda_{w},v_{w,v}\}_{w\in\Omega,v\in\Omega}$ be an continuous relay fusion frame for $H$ with continuous relay bounds $A$ and $B$. Suppose that $\{Z_{w,v}\}_{v\in\Omega_{w}}$ is a family of closed subspaces in $K_{w}$ for each $w\in\Omega$ and there exist constants $C,D,\epsilon$ such that $\max\{C+\frac{\epsilon}{\sqrt{A}},D\}<1$ and for all $f\in H$
	\begin{align*}
		\biggl(\int_{\Omega}\int_{\Omega_{w}}&\alpha_{w,v}^{2}\|P_{V_{w,v}}\Lambda_{w}\pi_{W_{w}}f-P_{Z_{w,v}}\Lambda_{w}\pi_{W_{w}}f\|^{2}d\mu(v)d\mu(w)\biggr)^{\frac{1}{2}}\\&\leq C\biggl(\int_{\Omega}\int_{\Omega_{w}}\alpha_{w,v}^{2}\|P_{V_{w,v}}\Lambda_{w}\pi_{W_{w}}f\|^{2}d\mu(v)d\mu(w)\biggr)^{\frac{1}{2}}\\&\quad+D\biggl(\int_{\Omega}\int_{\Omega_{w}}\alpha_{w,v}^{2}\|P_{Z_{w,v}}\Lambda_{w}\pi_{W_{w}}f\|^{2}d\mu(v)d\mu(w)\biggr)^{\frac{1}{2}}+\epsilon \|f\|.
	\end{align*}
Then $\mathcal{R}_{2}=\{W_{w},Z_{w,v},\Lambda_{w},\alpha_{w,v}\}_{w\in\Omega,v\in\omega}$ is an continuous relay fusion frame for $H$ with continuous relay fusion frame bounds
\begin{equation*}
	A\biggl(\frac{1-C-\frac{\epsilon}{\sqrt{A}}}{1+D}\biggr)^{2},\quad B\biggl(\frac{1+A+\frac{\epsilon}{\sqrt{B}}}{1-D}\biggr)^{2}.
\end{equation*}
\end{theorem}
\begin{proof}
	We have for each $f\in H$
	\begin{align*}
		\biggl(\int_{\Omega}\int_{\Omega_{w}}&\alpha_{w,v}^{2}\|P_{V_{w,v}}\Lambda_{w}\pi_{W_{w}}f\|^{2}d\mu(v)d\mu(w)\biggr)^{\frac{1}{2}}-\biggl(\int_{\Omega}\int_{\Omega_{w}}\alpha_{w,v}^{2}\|P_{Z_{w,v}}\Lambda_{w}\pi_{W_{w}}f\|^{2}d\mu(v)d\mu(w)\biggr)^{\frac{1}{2}}\\&\leq \biggl(\int_{\Omega}\int_{\Omega_{w}}\alpha_{w,v}^{2}\|P_{V_{w,v}}\Lambda_{w}\pi_{W_{w}}f-P_{Z_{w,v}}\Lambda_{w}\pi_{W_{w}}f\|^{2}d\mu(v)d\mu(w)\biggr)^{\frac{1}{2}}\\&\leq \bigg(C+\frac{\epsilon}{\sqrt{A}}\bigg)\biggl(\int_{\Omega}\int_{\Omega_{w}}\alpha_{w,v}^{2}\|P_{V_{w,v}}\Lambda_{w}\pi_{W_{w}}f\|^{2}d\mu(v)d\mu(w)\biggr)^{\frac{1}{2}}\\&\quad \quad+D\biggl(\int_{\Omega}\int_{\Omega_{w}}\alpha_{w,v}^{2}\|P_{Z_{w,v}}\Lambda_{w}\pi_{W_{w}}f\|^{2}d\mu(v)d\mu(w)\biggr)^{\frac{1}{2}}.
	\end{align*}
Then 
\begin{equation*}
A\biggl(\frac{1-C-\frac{\epsilon}{\sqrt{A}}}{1+D}\biggr)^{2}\|f\|^{2}\leq \int_{\Omega}\int_{\Omega_{w}}\alpha_{w,v}^{2}\|P_{Z_{w,v}}\Lambda_{w}\pi_{W_{w}}f\|^{2}d\mu(v)d\mu(w).
\end{equation*}
Similarly we can prove that 
\begin{equation*}
	B\biggl(\frac{1+C+\frac{\epsilon}{\sqrt{B}}}{1-D}\biggr)^{2}\|f\|^{2}\geq \int_{\Omega}\int_{\Omega_{w}}\alpha_{w,v}^{2}\|P_{Z_{w,v}}\Lambda_{w}\pi_{W_{w}}f\|^{2}d\mu(v)d\mu(w).
\end{equation*}
\end{proof}
\begin{theorem}
	Let $\mathcal{R}_{1}=\{W_{w},V_{w,v},\Lambda_{w},\alpha_{w,v}\}_{w\in\Omega,v\in\Omega_{w}}$ be an continuous relay fusion frame for $H$ with continuous relay fusion frame bounds $A$ and $B$. Suppose that $\{Z_{w,v}\}_{v\in\Omega_{w}}$ is a family of closed subspaces in $K_{w}$ for each $w\in\Omega$ for each $w\in\Omega$ and there exists a constant $0<C<A$ such that 
	\begin{equation*}
		\int_{\Omega}\int_{\Omega_{w}}\alpha_{w,v}^{2}\|P_{V_{w,v}}\Lambda_{w}\pi_{W_{w}}f-P_{Z_{w,v}}\Lambda_{w}\pi_{W_{w}}f\|^{2}d\mu(v)d\mu(w)\leq C\|f\|^{2},\quad\forall f\in H.
	\end{equation*}
Then $\mathcal{R}_{2}=\{W_{w},Z_{w,v},\Lambda_{w},\alpha_{w,v}\}_{w\in\Omega,v\in\Omega_{w}}$ is an continuous relay fusion frame for $H$ with continuous relay fusion frame bounds 
\begin{equation*}
	\sqrt{C}-\sqrt{A},\quad \quad \sqrt{C}+\sqrt{B}.
\end{equation*}
\end{theorem}
\begin{proof}
	For each $f\in H$, we have 
	\begin{align*}
		\|\{\alpha_{w,v}&P_{Z_{w,v}}\Lambda_{w}\pi_{W_{w}}f\}_{w\in\Omega,v\in\Omega_{w}}\|\\&\leq \|\{\alpha_{w,v}P_{V_{w,v}}\Lambda_{w}\pi_{W_{w}}f-\alpha_{w,v}P_{Z_{w,v}}\Lambda_{w}\pi_{W_{w}}f\}_{w\in\Omega,v\in\Omega_{w}}\|+\|\{\alpha_{w,v}P_{V_{w,v}}\Lambda_{w}\pi_{W_{w}}f\}_{w\in\Omega,v\in\Omega_{w}}\|\\&\leq (\sqrt{C}+\sqrt{B})\|f\|,
	\end{align*}
then
\begin{equation*}
	\int_{\Omega}\int_{\Omega_{w}}\alpha_{w,v}^{2}\|P_{Z_{w,v}}\Lambda_{w}\pi_{W_{w}}f\|^{2}d\mu(v)d\mu(w)\leq (\sqrt{C}+\sqrt{B})^{2}\|f\|^{2}.
\end{equation*}
Similarly we have for each $f\in H$
\begin{equation*}
\int_{\Omega}\int_{\Omega_{w}}\alpha_{w,v}^{2}\|P_{Z_{w,v}}\Lambda_{w}\pi_{W_{w}}f\|^{2}d\mu(v)d\mu(w)\geq (\sqrt{C}-\sqrt{A})^{2}\|f\|^{2}.	
\end{equation*}

\end{proof}

		\section{Acknowledgments}
	It is our great pleasure to thank the referee for his careful reading of the paper and for several helpful suggestions.
	
	\medskip
	
	\section*{Declarations}
	
	\medskip
	
	\noindent \textbf{Availablity of data and materials}\newline
	\noindent Not applicable.
	
	\medskip

	\noindent \textbf{Competing  interest}\newline
	\noindent The authors declare that they have no competing interests.

	\medskip
	
	\noindent \textbf{Fundings}\newline
	\noindent  Authors declare that there is no funding available for this article.

	\medskip
	
	\noindent \textbf{Authors' contributions}\newline
	\noindent The authors equally conceived of the study, participated in its
	design and coordination, drafted the manuscript, participated in the
	sequence alignment, and read and approved the final manuscript. 
	
	\medskip
	\bibliographystyle{amsplain}

\end{document}